\documentclass[12pt]{article}
\usepackage{latexsym, amsmath, amsfonts, amsthm, amssymb}
\usepackage{times}
\usepackage{a4wide}
\usepackage{times}
\usepackage{graphicx, tocloft}
\usepackage{mathrsfs}
\usepackage{chngcntr}

\counterwithout{subsubsection}{section}

\def \RR {\mathbb R}

\def \EE {\mathbb E}

\def \PP {\mathbb P}

\def \eps {\varepsilon}

\newtheorem{theorem}{Theorem}

\newtheorem{lemma}[theorem]{Lemma}

\newtheorem{proposition}[theorem]{Proposition}
\newtheorem{corollary}[theorem]{Corollary}

\theoremstyle{definition}
\newtheorem{remark}[theorem]{Remark}

\def\myffrac#1#2 in #3{\raise 2.6pt\hbox{$#3 #1$}\mkern-1.5mu\raise 0.8pt\hbox{$
		#3/$}\mkern-1.1mu\lower 1.5pt\hbox{$#3 #2$}}

\def\qed{\hfill $\vcenter{\hrule height .3mm
		\hbox {\vrule width .3mm height 2.1mm \kern 2mm \vrule width .3mm
			height 2.1mm} \hrule height .3mm}$ \bigskip}

\begin{document}

\author{Bo'az Klartag\textsuperscript{1} and Sasha Sodin\textsuperscript{2}}
\footnotetext[1]{Department of Mathematics, Weizmann Institute of Science, Rehovot 76100, Israel.
Email: boaz.klartag@weizmann.ac.il.  Supported by a grant from the Israel Science Foundation (ISF).}
\footnotetext[2]{School of Mathematical Sciences,
Queen Mary University of London,
Mile End Road, London E1 4NS, UK. Email: a.sodin@qmul.ac.uk. Supported in part by a Royal Society Wolfson Research Merit Award (WM170012), and a Philip Leverhulme Prize of the Leverhulme Trust (PLP-2020-064).}

\title{Local tail bounds for polynomials on the discrete cube}
\date{}
\maketitle

\begin{abstract} Let $P$ be a polynomial of degree $d$ in independent Bernoulli random variables which has zero mean and unit variance. The Bonami hypercontractivity bound implies that the probability that $|P| > t$ decays exponentially in $t^{2/d}$. Confirming a conjecture of Keller and Klein, we prove a local version of this bound, providing an upper bound on the difference between the $e^{-r}$ and the $e^{-r-1}$ quantiles of $P$.
\end{abstract}

\subsubsection{} This note is concerned with concentration inequalities for polynomials on the discrete cube. Concentration inequalities, i.e. tail bounds on the distribution of functions on high-dimensional spaces belonging to certain classes, were put forth by Vitali Milman in the 1970-s and have since found numerous applications; see e.g.\  \cite{GM, O}  and references therein.

\medskip
Let $X_1,\ldots,X_n$ be independent, identically distributed symmetric Bernoulli variables, so that $X = (X_1,\ldots,X_n)$ is  distributed uniformly on the discrete cube $\{ -1, 1 \}^n$. The starting point for this work is the concentration inequality for polynomials in $X$ (see e.g.\ \cite[Theorem 9.23]{O}), which we now recall.
Let $d \geq 1$, and consider a polynomial of the form
\begin{equation}  P_d(x) = \sum_{\#(S) = d } a_S \cdot \left( \prod_{i \in S} x_i \right) \label{eq_1713} \end{equation}
where the sum runs over all subsets $S \subseteq \{ 1, \ldots, n \}$ of cardinality $d$, and the coefficients $(a_S)$ are arbitrary
real numbers. In other words, $P_d$ is a $d$-homogeneous, square-free polynomial in $\RR^n$.
The Bonami hypercontractivity theorem \cite[Chapter 9]{O} tells us that for any $1 < p \leq q$,
\begin{equation} \| P_d(X) \|_q \leq \left( \frac{q-1}{p-1} \right)^{d/2} \| P_d(X) \|_p.
\label{eq_1702} \end{equation}
A general polynomial $P$ of degree at most $d$ on $\{-1,1 \}^n$ takes the form
\begin{equation} P(x) = \sum_{k=0}^d P_k(x) \label{eq_215} \end{equation}
where $P_k$ is a $k$-homogeneous, square-free polynomial. Thanks to orthogonality relations we have
$$ \| P(X) \|_2^2 = \sum_{k=0}^d \| P_k(X) \|_2^2. $$
Hence, by the Bonami bound (\ref{eq_1702}) and the Cauchy-Schwarz inequality, for any polynomial $P$ of degree at most $d$ and any $q \geq  3$,
\begin{align} \nonumber \| P(X) \|_q & \leq \sum_{k=0}^d \| P_k(X) \|_q \leq \sum_{k=0}^d (q-1)^{k/2} \| P_k(X) \|_2 \leq \sqrt{ \sum_{k=0}^d (q-1)^k } \cdot \sqrt{ \sum_{k=0}^d \| P_k(X) \|_2^2 } \\ & \leq \sqrt{2} \cdot (q-1)^{d/2} \cdot \| P(X) \|_2 \leq \sqrt{2} q^{d/2} \|P(X)\|_2 . \label{eq_1727} \end{align}
For $r > 0$ (not necessarily integer), write $a_r$ for a $e^{-r}$-quantile of $P(X)$, i.e.\ a number satisfying
$$ \PP( P(X) \geq a_r ) \geq \frac{1}{e^r} \qquad \text{and also} \qquad \PP( P(X) \leq a_r ) \geq 1- \frac{1}{e^{r}}. $$
Assume the normalization $\| P(X) \|_2 = 1$. It follows from (\ref{eq_1727}) that if $q \geq 3$ then
$$ \frac{1}{e^r} \leq  \PP( P(X) \geq a_r ) \leq  \frac{\EE |P(X)|^q}{a_r^q} \leq  \left( \sqrt 2 \cdot \frac{q^{d/2}}{a_r} \right)^q. $$
Substituting $q = 2r/d$ (when $r \geq 3d/2$), we get
\begin{equation}  a_r \leq \sqrt 2 \cdot (2er/d)^{d/2} \leq (C r/d)^{d/2} \quad ( r \geq 3d/2),\label{eq_1703}
\end{equation}
with a universal constant $C =4$. Without assuming any normalisation, we obtain
\begin{equation}  a_r  - a_1 \leq  C^d \left(\frac r d + 1\right)^{d/2}  \|P(X)\|_2~.\label{eq_1703'}
\end{equation}
(with a different numerical constant $C>0$),
which is valid for all $r \geq 1$.

\medskip
The estimate   (\ref{eq_1703'}) is a a tail bound for the distribution of $P(X)$, i.e.\ concentration inequality. We refer to \cite{GM} and references therein for background on concentration inequalities, particularly, for polynomials, and to \cite{O} for applications of (\ref{eq_1703'}) .

\medskip
In some applications, it is important to have bounds on $a_s - a_r$ when $s \geq r$ are close to one another, e.g.\ $s = r+1$. Such bounds are called {\em local} tail bounds; see \cite{DL} and references therein. The following proposition, confirming a conjecture of Nathan Keller and Ohad Klein, provides a  local  version of (\ref{eq_1703'}). In the case $d=1$, it follows from the results in the aforementioned work \cite{DL}.

\begin{proposition} Let $P$ be a polynomial of degree at most $d$ on $\{ -1,1 \}^n$. Then for all $r \geq 1$,
\begin{equation}\label{eq_1703''}  a_{r+1} - a_r  \leq
C^d \left( \frac{r}{d}+1\right)^{\frac d 2 -1} \|P(X)\|_2  ~,\end{equation}
where $C>0$ is a universal constant.
 \label{prop_957}
\end{proposition}

\noindent Clearly, (\ref{eq_1703''}) implies (\ref{eq_1703'}). The estimate (\ref{eq_1703''}) gives the right magnitute of $a_r - a_{r+1}$, say, for
\begin{equation}\label{eq:pol-clt} P(X) = (X_1 + \cdots + X_n)^d~, \quad n \gg 1~.\end{equation}

\subsubsection{}
We now turn to the proof of Proposition~\ref{prop_957}. Write $\partial_i P$ for the partial derivative of $P$ with respect to the $i^{th}$ variable.
Thus $$ \partial_i P (x) = \frac{P(T_i^1 x) - P(T_i^{-1} x)}{2} \qquad \qquad \text{for} \ x \in \{-1,1 \}^n, $$ where $T_i^j$ is
the map that sets the $i^{th}$-coordinate of $x$ to the value $j$, and keeps the other coordinates intact.
Observe that $\partial_i P$ is a polynomial of degree at most $d-1$ if $P$ is of degree $d$. We denote by $\nabla P$ the vector function with coordinates $\partial_i P$. The first step in the proof of Proposition~\ref{prop_957} is to sharpen the quantile bound (\ref{eq_1703}).

\begin{lemma} Let $P$ be a polynomial of degree at most $d$ with $\EE |P(X)|^2 = 1$.
Then for any non-empty  subset $A \subseteq \{-1, 1 \}^n$ of relative size $\eps = \#(A) / 2^n$ we have
\begin{equation}  \frac{1}{\#(A)} \sum_{x \in A} |P(x)|^2 \leq C^d \cdot \max \left \{1, \left( \frac{ |\log \eps| }{d} \right)^{d} \right \}, \label{eq_955} \end{equation}
and
\begin{equation}  \frac{1}{\#(A)} \sum_{x \in A} |\nabla P(x)|^2 \leq
C^d \cdot \max \left \{1, \left( \frac{ |\log \eps| }{d} \right)^{d-1} \right \}~,\label{eq_1001}
\end{equation}
for a universal constant $C > 0$. \label{lem_955}
\end{lemma}

\begin{proof} Let $q \geq 3$. By H\"older's inequality followed by an application of (\ref{eq_1727}),
\begin{align*}  \sum_{x \in A} |P(x)|^2 & \leq (\#(A))^{1 - 2/q} \cdot \left( \sum_{x \in A} |P(x)|^q \right)^{2/q} = (\#(A))^{1 - 2/q} \cdot 2^{2n/q} \cdot \| P(X) \|_q^2
\\ & \leq (\#(A))^{1 - 2/q} \cdot 2^{2n/q} \cdot 2 q^d~,\end{align*}
whence
\[ \frac{1}{\#(A)} \sum_{x \in A} |P(x)|^2 \leq 2\eps^{-2/q} q^d~.\]
The estimate (\ref{eq_955}) clearly holds for $\eps \geq e^{-\frac{3d}{2}}$, therefore we assume that $\eps < e^{-\frac{3d}{2}}$.
Set
\[ q =   2 |\log \eps| / d  \geq 3 \]
and obtain
$$ \frac{1}{\#(A)} \sum_{x \in A} |P(x)|^2 \leq  \left( \frac{C}{d} \right)^{d} |\log \eps|^d. $$
This proves (\ref{eq_955}).
Since $\partial_i P $ is a polynomial of degree at most $d-1$, from (\ref{eq_955}),
$$
\frac{1}{\#(A)} \sum_{x \in A} |(\partial_i P)(x)|^2 \leq
C^d \cdot \max \left \{1, \left( \frac{ |\log \eps| }{d} \right)^{d-1} \right \}
\cdot \EE |(\partial_i P)(X)|^2~,$$
whence
$$
\frac{1}{\#(A)} \sum_{x \in A} |(\nabla P)(x)|^2 \leq
C^d \cdot \max \left \{1, \left( \frac{ |\log \eps| }{d} \right)^{d-1} \right \}
\cdot \EE |(\nabla P)(X)|^2~.$$
We decompose $P(X) = \sum_{k=0}^d P_k(X)$ as in (\ref{eq_215}), and use
the orthogonality relations
$$ \EE |\nabla P(X)|^2 = \sum_{k=0}^d \EE |\nabla P_k(X)|^2 = \sum_{k=0}^d k \cdot \EE |P_k(X)|^2 \leq d \cdot \EE |P(X)|^2 = d. $$
This proves (\ref{eq_1001}).
\end{proof}

Note that for any $f: \{-1,1 \}^n \rightarrow \RR$,
\begin{equation} \sum_{x \in \{-1,1 \}^n} |\nabla f(x)|^2 \leq 2 \cdot \sum_{x \in \{-1,1 \}^n} |\nabla f(x)|^2 \cdot 1_{ \{ f(x) \neq 0 \} }.
\label{eq_1015} \end{equation}
Indeed, the expression on the left-hand side of (\ref{eq_1015}) is  the sum over all oriented edges $(x,y) \in E$ in the Hamming
cube of the squared difference $|f(x) - f(y)|^2 / 4 $. This is clearly at most twice the sum over all
oriented edges $(x,y) \in E$ of the quantity $|f(x) - f(y)|^2 \cdot 1_{ \{f(x) \neq 0 \}} / 4$.

\medskip
Recall the log-Sobolev inequality (e.g. \cite[Chapter 10]{O}) which states that for any function $f : \{-1,1 \}^n \rightarrow \RR$,
\begin{equation} \EE f^2(X) \log f^2(X) - \EE f^2(X) \cdot \log \EE f^2(X) \leq 2 \EE |\nabla f(X)|^2. \label{eq_1013} \end{equation}
Moreover, let $A \subseteq \{ -1,1 \}^n$ be a non-empty set and denote $\eps = \#(A) / 2^n$. If the function $f$
is supported in $A$ and is not identicaly zero, then denoting $g = f / \sqrt{ \EE f^2(X) }$,
\begin{equation} \begin{split} \EE f^2(X) \log f^2(X) - \EE f^2(X) \cdot \log \EE f^2(X)  &= \EE f^2(X) \cdot \EE g^2(X) \log g^2(X) \\
&\geq \EE f^2(X) \cdot |\log \eps|, \label{eq_1049} \end{split}\end{equation}
because $g^2$ is supported in $A$, and  among all probability distributions supported in $A$,
the maximal entropy is attained for the uniform distribution.

\begin{proof}[Proof of Proposition \ref{prop_957}] Without loss of generality $\|P(X)\|_2 = 1$. We may assume that $a_{r+1} > a_r$, as otherwise there is nothing to prove.
Let $U = \{ x \in \{ -1, 1 \}^n \, ; \, f(x) > a_r \}$
and set $\eps = \#(U) / 2^n$. Then $e^{-(r+1)} \leq \eps \leq e^{-r}$, by the definition of the quantiles
$a_r$ and $a_{r+1}$. Denote $\chi(t) = \max(t - a_r, 0)$; this is a $1$-Lipschitz function on the real line.
Applying the log-Sobolev inequality (\ref{eq_1013}) to the function $h = \chi \circ P : \{-1,1\}^n \rightarrow \RR$ we get
\begin{equation} \EE h^2(X) \log h^2(X) - \EE h^2(X) \cdot \log \EE h^2(X) \leq 2 \EE |\nabla h|^2(X). \label{eq_238} \end{equation}
Since $h$ is supported in $U$, with $\eps = \#(U) / 2^n$, by (\ref{eq_1049}) and (\ref{eq_238}),
$$ \EE h^2(X) \cdot |\log \eps| \leq 2 \EE |\nabla h|^2(X) \leq 4 \EE |\nabla h(X)|^2 \cdot 1_{\{ h(X) > 0 \}}. $$
The last passage is the content of (\ref{eq_1015}).
Since $\chi$ is $1$-Lipschitz, we know that $|\nabla h|^2 \leq |\nabla P|^2$.
Hence, by (\ref{eq_1001}),
$$ \EE |\nabla h(X)|^2 \cdot 1_{\{ h(X) > 0 \}} \leq \EE |\nabla P(X)|^2 1_{\{ X \in U \}} \leq \eps \cdot C^d \cdot \max \left \{1, \left( \frac{ |\log \eps| }{d} \right)^{d-1} \right \}. $$
To summarize,
\begin{equation}  \EE h^2(X) \cdot |\log \eps| \leq
\eps \cdot C_1^d \cdot \max \left \{1, \left( \frac{ |\log \eps| }{d} \right)^{d-1} \right \}, \label{eq_245} \end{equation}
for a universal constant $C_1> 0$.
 Recall that $e^{-(r+1)} \leq \eps \leq e^{-r}$.
By the definition of $a_{r+1}$, we know that $h(X) \geq a_{r+1} - a_r$ with probability at least $e^{-(r+1)}$.
Therefore, from (\ref{eq_245}),
$$ e^{-(r+1)} \cdot (a_{r+1} - a_r)^2 \cdot \frac{r}{2} \leq e^{-r} \cdot C_1^d \cdot \max \left \{1, \left( \frac{2 r}{d} \right)^{d-1} \right \} $$
or
\[  a_{r+1} - a_r  \leq  C_2^d \cdot \max \left \{ \frac{1}{\sqrt{r}} , \left( \frac{ r }{d } \right)^{d/2-1} \right \} \leq C_3^d \left( \frac r d + 1 \right)^{\frac d 2}~.  \tag*{\qedhere} \]
\end{proof}

\subsubsection{}
We remark that   Proposition \ref{prop_957} implies the following corollary which holds true without the normalization by $\|P(X)\|_2$.

\begin{corollary}\label{cor} There exists $C > 0$ such that the following holds. Let $P$ be a polynomial of degree at most $d$ with $\mathbb E P(X) = 0$.  Then
for $r \geq Cd$, 
\begin{equation}\label{eq:concl-cor'}
a_{r+1} \leq  a_r \left[ 1 + C^d \left( \frac r d + 1 \right)^{\frac{d}{2}-1} \right]. 
\end{equation}
\end{corollary}
\begin{remark}
We conjecture that  (\ref{eq:concl-cor'}) also holds with the power $-1$ in place of $\frac d2 - 1$. Such an estimate would give the right order of magnitude for the polynomial (\ref{eq:pol-clt}).
\end{remark}

\begin{proof}[Proof of Corollary~\ref{cor}]
Write $\sigma^2 = \mathbb E |P(X)|^2$. We shall prove that $\sigma \leq C_1^d a_r$. Once this inequality is established, we deduce   from Proposition \ref{prop_957}  that
$$ \frac{a_{r+1} - a_r}{\sigma} \leq C^d  \left( \frac r d + 1 \right)^{\frac{d}{2}-1}~, $$
whence
$$ a_{r+1} \leq a_r + \sigma \cdot C^d  \left( \frac r d + 1 \right)^{\frac{d}{2}-1}
\leq  a_r \left(1 + (CC_1)^d \left( \frac r d + 1 \right)^{\frac{d}{2}-1} \right)~,$$
as claimed.

Let $\sigma_{\pm} = \sqrt{ \mathbb E (P(X)_\pm)^2 }$. First, we claim that $\sigma_+ \geq C_2^{-d} \sigma$. Indeed, if $\sigma_+ \geq \sigma_-$
then $\sigma_+ \geq \sigma / \sqrt{2}$. If $\sigma_+ < \sigma_-$, then, using (\ref{eq_1727}),
\[\begin{split}
\sigma_+ &\geq \mathbb E P(X)_+ = \mathbb E P(X)_- \geq \frac{(\mathbb E P(X)_-^2)^{3/2}}{(\mathbb E P(X)_-^4)^{1/2}} \\
&\geq \frac{\sigma_-^3}{2 \cdot 3^d \cdot (\sigma_+^2 + \sigma_-^2)} \geq \frac{1}{4 \cdot 3^d} \sigma_-~.
\end{split}\]
Second, another application of (\ref{eq_1727}) yields
\[ \mathbb E P(X)_+^4 \leq \mathbb E P(X)^4 \leq 4 \cdot 3^d  \sigma^4 \leq C_3^d \sigma_+^4~, \]
thus by the Paley--Zygmund inequality
\[ e^{-Cd} \geq e^{-r} \geq \mathbb P \left\{ P(X) > a_r  \right\} \geq \frac{(1 - a_r^2 / \sigma_+^2)_+^2}{C_3^d}~,\]
whence $\sigma_+ \leq 2 a_r$ if we ensure that, say,  $e^C \geq 2 C_3$. This concludes the proof.
\end{proof}

Finally, we remark that both  Proposition \ref{prop_957} and Corollary \ref{cor} can be generalised in several directions. For example, instead of the Hamming cube, one can consider a general measure which is invariant under a Markov diffusion satisfying the Bakry--\'Emery CD($R, \infty)$ condition; in this setting, linear combinations of eigenfunctions of the generator play the r\^ole of polynomials. The proof requires only notational modifications.

\paragraph{Acknowledgement.} We are grateful to Nathan Keller for helpful correspondence.

\end{document}